\documentclass{article}
\usepackage{amsmath,amsfonts,amsthm,amssymb}
\DeclareMathSymbol\nullset{\mathord}{AMSb}{"3F}

\def\set#1\endset{\{\,#1\,\}}
\def\from{\mkern2mu\hbox{\rm :}\mkern2mu}
\def\comp{\mkern2mu\mathchoice%
        {\raise.35ex\hbox{$\scriptscriptstyle\circ$}}
        {\raise.35ex\hbox{$\scriptscriptstyle\circ$}}
        {\raise.14ex\hbox{$\scriptscriptstyle\circ$}}
        {\raise.14ex\hbox{$\scriptscriptstyle\circ$}}}
\let\cong=\equiv
\def\mod#1{\,\,(\text{mod}\,#1)}

\def\choice#1,#2{\binom{#1}{#2}}
\def\kzero#1{k\langle {#1}\rangle_0}
\def\ckzero#1{k[#1]_0}
\def\class#1{x^{[#1]}}
\def\w#1,#2{W_{#1}(#2)}
\def\u#1,#2{U_{#1}(#2)}
\def\ybasis#1{B_{#1}}
\newcounter{parts}

\begin{document}

 \newtheorem{theorem}{Theorem}[section]
 \newtheorem{corollary}{Corollary}[section]
 \newtheorem{lemma}{Lemma}[section]
 \newtheorem{proposition}{Proposition}[section]
 {\newtheorem{definition}{Definition}[section]}
% \def\begin{proof}{\ifdim\lastskip<\smallskipamount\relax\removelastskip
%  \vskip\smallskipamount\fi\leavevmode\noindent\hbox to 0pt{\hfil}{\it Proof.}}
% \def\strutdepth{\dp\strutbox}
% \def\epmarker{\vbox to \strutdepth{\baselineskip\strutdepth\vss\hfill{%
% \hbox to 0pt{\hss\vrule height 4pt width 4pt depth 0pt}\null}}}
% \def\edproofmarker{\strut\vadjust{\kern-2\strutdepth\epmarker}}
%\def\endproof{\edproofmarker\vskip10pt}

%\begin{frontmatter}

\title{Maximal $T$-spaces of the free associative algebra over a finite field}
\author{C. Bekh-Ochir and S. A. Rankin}
%\ead{srankin@uwo.ca}
%\author[cbekhoch]{C. Bekh-Ochir}
%\ead{cbekhoch@gmail.com}
%\address[rankin,cbekhoch]{University of Western Ontario}
\maketitle
\begin{abstract}
  In earlier work, it was established that for any finite field $k$, the  
  free associative $k$-algebra on one generator $x$, denoted by $\ckzero{x}$,
  had infinitely many maximal $T$-spaces, but
  exactly two maximal $T$-ideals (each of which is a maximal $T$-space).
  However, aside from these two $T$-ideals, no examples of maximal
  $T$-spaces of $\ckzero{x}$ have been identified. This paper presents, 
  for each finite field $k$, an infinite sequence of proper $T$-spaces of 
  $\ckzero{x}$ (no one of which is a $T$-ideal), each of finite codimension, 
  and for each one, both a linear basis for the $T$-space itself and a linear basis for a
  complementary linear subspace are provided. Morever, it is proven that
  the first $T$-space in the sequence is a maximal $T$-space of
  $\ckzero{x}$, thereby providing the first example of a maximal $T$-space
  of $\ckzero{x}$ that is not a maximal $T$-ideal.
\end{abstract}

%\begin{keyword}
%  associative algebra\sep polynomial identities\sep $T$-space\sep maximal
%  
%  \MSC[2010] 16R10
%\end{keyword}

%\end{frontmatter}

\section{Introduction}
 Let $k$ be a field, and let $A$ be an associative $k$-algebra. 
 A\hbox{.} V\hbox{.} Grishin introduced the concept of a $T$-space  of
 $A$ (\cite{grI}, \cite{grII}); namely, a linear subspace of $A$ that is
 invariant under the natural action of the transformation monoid  $T$ of
 all $k$-algebra endomorphisms of $A$. A $T$-space of $A$ that is also an
 ideal of  $A$ is called a $T$-ideal of $A$.  For any
 $H\subseteq A$, the smallest $T$-space  of $A$ containing $H$ shall be
 denoted by $H^S$, while the smallest $T$-ideal of $A$ that  contains
 $H$ shall be denoted by $H^T$. The set of all $T$-spaces of $A$ forms a
 lattice under the inclusion ordering.
 
 We shall let $\kzero{X}$ denote the free associative $k$-algebra on a set $X$. Our interest in this
 paper shall be the study of the maximal elements in the lattice $L(\kzero{X})$ for $X$ any nonempty set. 
 It was shown in \cite{rankin} that if $k$ is infinite, then the unique maximal $T$-ideal of $\kzero{X}$
 (more precisely, there is a maximum $T$-ideal) is also the unique maximal
 $T$-space, while the story for $k$ finite was strikingly different. It turned out
 that when $k$ is finite, there are two maximal $T$-ideals, each of which is a also a maximal $T$-space,
 but now there are infinitely many maximal $T$-spaces of
 $\kzero{X}$. This was established by showing that there
 is a natural bijection between the sets of maximal $T$-spaces of
 $\kzero{X}$ and of $\ckzero{x}$, and then proving the result for $\ckzero{x}$. 
 
 While the approach taken in \cite{rankin} treated the cases $p>2$ and $p=2$ separately, 
 in each case an infinite family of $T$-spaces was constructed with the property that
 no maximal $T$-space of $\ckzero{x}$ could contain more than one of the constructed
 $T$-spaces. It was not proven in \cite{rankin} that any of the constructed $T$-spaces
 was in fact maximal, and it has turned out that the maximal $T$-spaces of $\ckzero{x}$
 (other than the maximum $T$-ideal) are elusive creatures.

 Our objective in this paper is to present, for any finite field $k$, another
 infinite sequence of $T$-spaces of $\ckzero{x}$ with the hope that each
 member of the sequence is maximal. Each of these $T$-spaces has finite codimension,
 and for each of these $T$-spaces, we are able to provide both a linear basis for the $T$-space
 and a linear basis for a complementary linear subspace of $\ckzero{x}$. Moreover,
 we shall prove that the first $T$-space in the sequence is maximal.
 
 Throughout the paper, $k$ shall denote an arbitrary field of
 order $q$ and characteristic $p\ge2$.
 
 Let $X$ be any nonempty set. In $\kzero{X}$, if $|X|=1$, let $T^{(2)}=\set 0\endset$, 
 and $Z_X=\set x^2\endset^T$, where $X=\set x\endset$, otherwise let
 $x,y\in X$ with $x\ne y$ and set $T^{(2)}=\set [x,y]\endset^T_X$, and
 $Z_X=\set xy\endset^T$. For any $x\in X$, let $W=T^{(2)}+\set x-x^q\endset^T_X$.

 For any finite field $k$, and any nonempty set $X$,
 $Z$ and $W$ are maximal $T$-ideals of $\kzero{X}$, and these are 
 the only maximal $T$-ideals of $\kzero{X}$. It was established in \cite{rankin} that each 
 is a maximal $T$-space of $\kzero{X}$. As well, it was established that for $x\in X$, 
 the map $\pi\from L(\kzero{X})\to L(\ckzero{x})$
 that is determined by sending each $y\in X$ to $x$ induces a bijection from
 the set of maximal $T$-spaces of $\kzero{X}$ onto the set of maximal $T$-spaces of
 $\ckzero{x}$. This established that every  maximal $T$-space of
 $\kzero{X}$ is uniquely determined by its one-variable polynomials.

 The following notion will be of fundamental importance in our work.
 Recall that $k$ is a finite field of order $q$. 
 For monomials $u_i\in \kzero{X}$ and $\alpha_i\in k$, $1\le i\le t$,
 $f=\sum_{i=1}^t \alpha_i u_i$ shall be said to be $q$-homogeneous if 
 for each $x\in X$ and each $i,j$ with $1\le i,j\le t$,  $\deg_x(u_i)
 \cong \deg_x(u_j)\mod{\mkern 4mu q-1}$.

\section{A sequence $W_n$, $n\ge1$, of $T$-spaces of $\kzero{X}$}

\begin{definition}
 Let $X$ be a nonempty set, and let  $x\in X$.
 For each $n\ge1$, let $\w n,X$ denote the $T$-space of $\kzero{X}$ that is generated by
 $x+x^{q^n}$ and $x^{q^n+1}$; that is, $\w n,X=\set x + x^{q^n}\endset^S + \set x^{q^n+1}\endset^S$.
 As well, let $\u n,X=\set x-x^{q^{2n}}\endset^T$ in $\kzero{X}$. If $X$ is finite, say $X=\set x_1,x_2,\ldots,x_m\endset$,
 we shall write $\w n,{x_1,x_2,\ldots,x_m}$ and $\u n,{x_1,x_2,\ldots,x_m}$ for $\w n,X$ and $\u n,X$, respectively.
 Finally, if $X=\set x\endset$, we shall simply write $W_n$ and $U_n$ for $\w n,X$ and $\u n,X$, respectively.
\end{definition}

There is a very important observation that we may make about $U_n$ that will have interesting applications in the work to come.

\begin{lemma}\label{lemma: almost like 1}
 Let $x\in X$. Then for any $u\in \ckzero{x}\subseteq \kzero{X}$, $x^{q^{2n}-1}u\cong u\mod{U_n}$.
\end{lemma}

\begin{proof}
 It suffices to prove the result for $u=x^i$, $i\ge 1$.
 If $i=1$, the result follows from the definition of $U_n$. Suppose that $i\ge 2$. Then
 $x^{q^{2n}-1}x^i=x^{q^{2n}}x^{i-1}\cong xx^{i-1}=x^i\mod{U_{n}}$.
\end{proof}

\begin{lemma}\label{lemma: fundamental wn}
 Let $n\ge1$. Then for any $u,v\in \kzero{X}$, $uv^{q^n}+u^{q^n}v\in \w n,X$.
\end{lemma}

\begin{proof}
 For any $u,v\in \kzero{X}$, we have $(u+v)^{q^n+1}=(u+v)^{q^n}(u+v)=(u^{q^n}+v^{q^n})(u+v)
 =u^{q^n+1}+u^{q^n}v+v^{q^n}u+(uv)^{q^n}$. Since $(u+v)^{q^n+1}$, $u^{q^n+1}$, and $v^{q^n+1}$
 each belong to $\w n,X$, it follows that $u^{q^n}v+v^{q^n}u\in \w n,X$.
\end{proof}

\begin{lemma}
 For every $n\ge1$,  $\u n,X\subseteq \w n,X$.
\end{lemma}

\begin{proof}
 Let $u,v\in\kzero{X}$. Then $(u+u^{q^n})(v+v^{q^n})=uv+uv^{q^n}+u^{q^n}v+(uv)^{q^n}$.
 As $uv+(uv)^{q^n}\in \w n,X$ by definition, and by Lemma \ref{lemma: fundamental wn}, $uv^{q^n}+u^{q^n}v\in \w n,X$,
 it follows that $(u+u^{q^n})(v+v^{q^n})\in \w n,X$. Now note that $(u-u^{q^{2n}})(v+v^{q^n})=
 (u+u^{q^{n}})(v+v^{q^n})-(u^{q^n}+(u^{q^{n}})^{q^n})(v+v^{q^n})$, and thus, since
 $(u+u^{q^n})(v+v^{q^n})\in \w n,X$, and $(u^{q^n}+(u^{q^{n}})^{q^n})(v+v^{q^n})\in \w n,X$, we have
 $(u-u^{q^{2n}})(v+v^{q^n})\in \w n,X$. But 
 \begin{align*}
   (u-u^{q^{2n}})(v+v^{q^n})&=(u-u^{q^{2n}})v+ uv^{q^n}-u^{q^{2n}}v^{q^{2n}}\\
   &=(u-u^{q^{2n}})v+ uv^{q^n}+u^{q^n}v-(u^{q^n}v+(u^{q^{n}}v)^{q^n}).
 \end{align*} 
 By Lemma \ref{lemma: fundamental wn}, $uv^{q^n}+u^{q^n}v\in\w n,X$, and by definition, $u^{q^n}v+(u^{q^{n}}v)^{q^n}\in \w n,X$.
 As well, we have shown that $(u-u^{q^{2n}})(v+v^{q^n})\in \w n,X$, and so it follows that $(x-x^{q^{2n}})v\in \w n,X$.
\end{proof}

In the proof of the preceding lemma, we showed that $(u+u^{q^n})(v+v^{q^n})\in \w n,X$ for every $u,v\in \kzero{X}$.
We can say more in this regard. For any $u,v\in\kzero{X}$, we have 
$$
 (u+u^q)v^{q^n+1}=uv^{q^n+1}+u^{q^n}v^{q^n+1}\cong uv^{q^n+1}-u(v^{q^n+1})^{q^n}\mod{\w n,X}.
$$
As well, $u(v^{q^n+1})^{q^n}=uv^{q^{2n}+q^n}\cong uv^{1+q^n}\mod{\u n,X}$.
Since $\u n,X\subseteq \w n,X$, we have $(u+u^q)v^{q^n+1}\cong uv^{q^n+1}-uv^{q^n+1}=0\mod{\w n,X}$, so $(u+u^q)v^{q^n+1}\in 
\w n,X$. A similar argument shows that $v^{q^n+1}(u+u^q)\in \w n,X$. Thus for each $n\ge1$, $\w n,X$ is a subalgebra
of $\kzero{X}$.

We now explore more carefully the case when $X=\set x\endset$, in which case $\kzero{X}=\ckzero{x}$.

\begin{lemma}\label{lemma: un basis}
  The set $\set (x^{q^{2n}}-x) x^i \mid i\ge0\endset$ is a linear basis for $U_n$.
\end{lemma}

\begin{proof}
  For any $\alpha,\beta\in k$ and $v,w\in \ckzero{x}$,
  $(\alpha v+\beta w)^{q^{2n}}-(\alpha v+\beta w)=\alpha v^{q^{2n}}+\beta w^{q^{2n}}-\alpha v-\beta w
  =\alpha(v^{q^{2n}}-v)+\beta(w^{q^{2n}}-w)$.
  Consider $u\in U_n$. We have $u=\sum_{i=1}^t \alpha_i (u_i^{q^{2n}}-u_i)v_i$ for some $u_1,v_1,u_2,v_2,
  \ldots,u_t,v_t\in \ckzero{x}$ and $\alpha_1,\ldots,\alpha_t\in k$. By the above observation, we may assume that
  each $u_i$ is a monomial; that is,  we may assume that $u$ has the form
   $u=\sum_{i=1}^t \alpha_i( (x^{r_i})^{q^{2n}} - x^{r_i})v_i$ for positive
  integers $r_i$, $i=1,\ldots,t$.  For each $i$, by factoring as a difference of $r_i$ powers,
  we may write $(x^{r_i})^{q^{2n}} - x^{r_i}=(x^{q^{2n}} - x)w_i$, for some 
  $w_i\in \ckzero{x}$. For each $i$, $(x^{q^{2n}} - x)w_iv_i$ is in the linear space spanned by
  $\set (x^{q^{2n}} - x)x^j\mid j\ge0\endset$. Since $x^{q^{2n}} - x\in U_n$, we have $(x^{q^{2n}} - x)x^j\in U_n$ for
  each $j\ge1$, and so it follows that the set $\set (x^{q^{2n}}-x) x^i \mid i\ge0\endset$ is a
  spanning set for $U_n$. The linear independence is immediate since no two polynomials
  in the set have the same degree.
\end{proof}

The set $\set (x^{q^{2n}}-x) x^i \mid i\ge0\endset$ contains exactly one polynomial of each degree greater than
or equal to $q^{2n}$, and so the dimension of $\ckzero{x}/U_n$ as a $k$-vector space is $q^{2n}-1$.
Note that if $1\le m\le q^{2n}-1$, then by the division theorem, there exist unique integers $t$ and $r$ with
$m=tq^n+r$ and $0\le r<q^n$. Since $n\le q^{2n}-1$, we have $tq^n+r\le q^{2n}-1$, so $t\le q^n-(r+1)/q^n\le q^n-1/q^n$.
Since $t$ is an integer, it follows that $t\le q^n-1$, so we have $0\le t,r\le q^b-1$ and not both $t$ and $r$ can
be 0. The uniqueness of $t$ and $r$ establishes that no two polynomials in the set
$$
\set x^{q^ni+j}+x^{i+q^nj}\mid q^n>i> j\ge 0\endset \cup \set (x^{q^n+1})^i\mid 1\le i\le q^n-1\endset
$$
have the same degree, which establishes the following fact.

\begin{lemma}\label{lemma: fundamental set}
 The set 
 $$
  \set x^{q^ni+j}+x^{i+q^nj}\mid q^n>i> j\ge 0\endset \cup \set (x^{q^n+1})^i\mid 1\le i\le q^n-1\endset
 $$
 is linearly independent in $\ckzero{x}$.
\end{lemma}
 
\begin{definition}
 For each $n\ge1$, and $i,j$ with $0\le i,j<q^n$ and $i\ne j$, let $F(i,j)=x^{iq^n+j}+x^{i+jq^n}$, and let
 $F(i,i)=(x^{q^n+1})^i$ if $1\le i<q^n$. Then set
 $$
  E_n=\set F(i,j)\mid q^n>i> j\ge 0\endset \cup \set F(i,i)\mid 1\le i\le q^n-1\endset,
 $$
 and let $V_n$ denote the linear span of $E_n$ in $\ckzero{x}$.
\end{definition} 

It follows from Lemma \ref{lemma: fundamental set} that the dimension of $V_n$ (as a $k$-vector space) is $\choice q^n,2 + q^n-1$.
Furthermore, we note that if $0\le j<i<q^n$, then the degree of $F(i,j)=F(j,i)$ is $iq^n+j$.

Note that if $p>2$, the set $\set x^{q^ni+j}+x^{i+q^nj}\mid q^n>i\ge  j\ge 0,\ i+j>0\endset$ is a basis for $V_n$,
as taking $i=j$ in $x^{q^ni+j}+x^{i+q^nj}$ gives $2(x^{q^n+1})^i$.

\begin{proposition}\label{proposition: wn/un basis}
 For each $n\ge 1$, $W_n=V_n\oplus U_n$.
\end{proposition}

\begin{proof}
 Note that when $i>j=0$, then $x^{q^ni+j}+x^{i+q^nj}=x^{q^ni}+x^{i}\in
 W_n$, while if $i>j>0$,  $x^{q^ni+j}+x^{i+q^nj}\in W_n$ by virtue of
 Lemma \ref{lemma: fundamental wn}. Thus $V_n\subseteq W_n$.
 Furthermore, as the elements of $E_n$ have degree at most
 $q^n(q^n-1)+q^n-1=q^{2n}-1<q^{2n}$, no two elements of $E_n\cup\set
 (x^{q^{2n}}-x)x^i\mid i\ge 0\endset$ have the same degree, so
 $E_n\cup\set (x^{q^{2n}}-x)x^i\mid i\ge 0\endset$ is linearly independent 
 and $V_n\cap U_n=\set 0\endset$. It remains to prove that $E_n\cup\set
 (x^{q^{2n}}-x)x^i\mid i\ge 0\endset$ is a spanning set for $W_n$. 
 
 Observe that $(u+v)^{q^n+1}=u^{q^n+1}+v^{q^n+1}+(uv^{q^n}+u^{q^n}v)$,
 and the expression $(uv^{q^n}+u^{q^n}v)$ is linear in each of $u$ and
 $v$, so the set  $\set F(i,j)\mid i> j\ge 1\endset\cup \set F(i,i)\mid
 i\ge 1\endset$ is a basis for  $\set x^{q^n+1}\endset^S$, while the set
 $\set F(i,0)\mid i>0\endset$ is a basis for $\set x+x^{q^n}\endset^S$. 
 Thus the set $\set F(i,j)\mid i\ge j,\ i+j\ne 0\endset$ is a linear
 basis for $W_n$. It suffices therefore to prove that for each $i>0$,
 there exists $i_1$ with $q^n>i_1\ge1$ such that $F(i,i)\cong
 F(i_1,i_1)\mod{U_n}$, and for each $j$ with $i>j\ge0$, there exist
 $i_1,j_1$ with $q^n>i_1\ge j_1\ge0$ and $i_1+j_1>0$ such that
 $F(i,j)\cong F(i_1,j_1)\mod {U_n}$. This we do by induction on $i\ge1$.
 The assertion is obviously true for $1\le i\le q^n-1$, so we suppose
 that  $i\ge q^n$ is such that the assertion holds for all smaller
 integers. Let $t=i-q^n\ge0$. Then 
 $F(i,i)=(x^{q^n+1})^i=(x^{q^n+1})^{(t+q^n)}=x^{q^{2n}+q^nt+t+q^n}\cong
 x^{1+q^nt+t+q^n}=F(t+1,t+1)\mod{U_n}$, and $t+1<t+q^n=i$, so by the
 induction hypothesis, there exists $i_1<q^n$ such that $F(t+1,t+1)\cong
 F(i_1,i_1)\mod{U_n}$. But then $F(i,i)\cong F(t+1,t+1)\cong
 F(i_1,i_1)\mod{U_n}$, as required. Now let $0\le j<i$. Suppose first
 that $j\ge q^n$ as well. For $i=t+q^n$ and $j=r+q^n$, we have $F(i,j)=
 x^{(t+q^n)q^n+r+q^n}+x^{t+q^n+(r+q^n)q^n}
 =x^{tq^n+q^{2n}+r+q^n}+x^{t+q^n+rq^n+q^{2n}}\cong
 x^{tq^n+1+r+q^n}+x^{t+q^n+rq^n+1}=F(t+1,r+1)$. By the induction
 hypothesis, since $i>t+1>r+1\ge 0$, there exist $i_1,j_1$ with
 $q^n>i_1\ge j_1\ge 0$ and $i_1+j_1>0$ such that $F(i,j)\cong
 F(i_1,j_1)\cong F(t+1,r+1)\mod{U_n}$, as required. Suppose now that
 $j<q^n$. As before, set $i=t+q^n$, and consider $F(i,j)$. We have
 $F(i,j)=x^{(t+q^n)q^n+j}+x^{t+q^n+jq^n}=x^{tq^n+q^{2n}+j}+x^{t+q^n+jq^n}\cong
 x^{tq^n+1+j}+x^{t+q^n(j+1)}=F(t,j+1)\mod{U_n}$. Since $i>t$, the result
 follows from the inductive hypothesis if $t\ge j+1$, or if $t<j+1<i$.
 Suppose that $t<j+1=i$.  Since $j<q^n$ and $i\ge q^n$, we must have
 $i=q^n$ and $j=q^n-1$. But then $t=0$, and $F(t,j+1)=F(0,q^n)=x^{q^n}+x^{q^{2n}}
 \cong x^{q^n}+x=F(0,1)\mod{U_n}$, which completes the proof of the inductive step.
 Thus $E_n\cup\set (x^{q^{2n}}-x)x^i\mid i\ge 0\endset$ 
 is a spanning set for $W_n$.
\end{proof}

 We remark that in the proof of Proposition \ref{proposition: wn/un basis}, it was established
 that
 $$
   E_n\cup \set (x^{q^{2n}}-x) x^i \mid i\ge0\endset
 $$
 is a linear basis for $W_n$.

\begin{corollary}
 $\dim(\ckzero{x}/W_n)=\choice q^n,2$. In particular, $W_n$ is a proper $T$-space of $\ckzero{x}$.
\end{corollary} 

\begin{proof}
  The dimension of $\ckzero{x}/W_n$ is $q^{2n}-1 -(q^n(q^n-1)/2 +q^n-1)=q^{2n}/2-q^n/2=
  q^n(q^n-1)/2=\choice q^n,2$.
\end{proof}

\section{The maximality of $W_n$}

In this section, we begin to investigate the maximality of $W_n$ in $\ckzero{x}$ for $n\ge1$. 

We have seen that each integer $m$ with $1\le m\le q^{2n}-1$ is uniquely of the form $m=tq^n+r$
with $0\le t,r<q^n$ and $t+r>0$. Thus in the set $E_n\cup \set (x^{q^{2n}}-x) x^i \mid i\ge0\endset$, there are no
polynomials with degree of the form $jq^n+i$ with $q^n>i>j\ge0$. Consequently,
$$
 E_n\cup\set (x^{q^2n}-x)x^i\mid i\ge 0\endset\cup \set x^{i+q^nj}\mid q^n>i>j\ge 0\endset
$$
is linearly independent in $\ckzero{x}$, and contains polynomials of
each degree greater than or equal to 1, hence is a linear basis for $\ckzero{x}$. 
It follows that the set $\set x^{i+q^nj}\mid q^n>i>j\ge 0\endset$ containing $\choice q^n,2$ polynomials
is a $k$-linear basis for a subspace of $\ckzero{x}$ that is complentary to $W_n$.

\begin{definition}\label{definition: yn def}
 For each $n\ge 1$, let $\ybasis{n}=\set x^{i+q^nj}\mid q^n>i>j\ge 0\endset$, and let $Y_n$ denote the linear subspace of 
 $\ckzero{x}$ that is spanned by $\ybasis{n}$.
\end{definition}

Thus $\ckzero{x}=Y_n\oplus W_n=Y_n\oplus V_n\oplus U_n$. In order to establish that $W_n$ is maximal, it suffices to
show that for any nonzero $f\in Y_n$, $W_n+\set f\endset^S=\ckzero{x}$. Moreover, since each $q$-homogeneous component
of $f$ belongs to any $T$-space that contains $f$, it will suffice to prove that for any nonzero $q$-homogeneous polynomial 
$f\in Y_n$, $W_n+\set f\endset^S=\ckzero{x}$.

\begin{lemma}\label{lemma: basic for containment}
 For any positive integer $r$, the following hold in $\ckzero{x}$.
  \begin{list}{(\roman{parts})}{\usecounter{parts}} 
  \item $x^{q^{2^rm}}\cong x^{q^{2^r(m-2)}}\mod{U_{2^r}}$ for any $m\ge3$.
  \item  $x+(-1)^{m+1}x^{q^{2^rm}}\in \set x+x^{q^{2^r}}\endset^S$ for any $m\ge1$.
 \end{list}
\end{lemma}

\begin{proof}
 Let $m\ge 3$.
 We have $q^{2^rm}=q^{2^r(m-2)+2^{r+1}}=q^{2^r(m-2)}q^{2^{r+1}}$, and so
 $$
   x^{q^{2^rm}}=(x^{2^{r+1}})^{q^{2^r(m-2)}} \cong x^{q^{2^r(m-2)}}\mod{U_{2^r}},
 $$
 which establishes the first part. The second part is proven by induction on $m\ge1$, with
 the case for $m=1$ true by definition. Suppose that $m\ge1$ is an integer
 for which the result holds, so $x+(-1)^{m+1}x^{q^{2^rm}}\in \set x+x^{q^{2^r}}\endset^S$. Apply the
 substitution $x\mapsto x^{q^{2^rm}}$ to $x+x^{q^{2^r}}$ to obtain that $x^{q^{2^rm}}+x^{q^{2^rm}q^{2^r}}
 \in \set x+x^{q^{2^r}}\endset^S$. Thus $x+(-1)^{m+2}x^{q^{2^r(m+1)}}=
 x+(-1)^{m+1}x^{q^{2^rm}} + (-1)^{m+2}[x^{q^{2^rm}}+x^{q^{2^r(m+1)}}]
 \in \set x+x^{q^{2^r}}\endset^S$.  The result follows now by induction.
\end{proof}

\begin{proposition}
 For each $r\ge1$ and each odd $m\ge1$, $W_{2^rm}\subseteq W_{2^r}$.
\end{proposition}

\begin{proof}
 By Lemma \ref{lemma: basic for containment} (i) and induction on odd
 $m\ge1$, $x^{q^{2^rm}}\cong x^{q^{2^r}}\mod{U_{2^r}}$  for every odd
 $m\ge1$. Let $m\ge1$ be odd. Then $x^{q^{2^rm}+1}\cong
 x^{q^{2^r}+1}\mod{U_{2^r}}$. Since $x^{q^{2^r}+1}\in W_{2^r}$ and
 $U_{2^r}\subseteq W_{2^r}$, it follows that $x^{q^{2^rm}+1}\in
 W_{2^r}$. Next,  since $m$ is odd, it follows from Lemma \ref{lemma:
 basic for containment} (ii) that  $x+x^{q^{2^rm}}\in \set
 x+x^{q^{2^r}}\endset^S\subseteq W_{2^r}$. Thus $W_{2^rm}=\set
 x+x^{q^{2^rm}}\endset^S+\set x^{q^{2^rm}+1}\endset^S
 \subseteq \set x+x^{q^{2^r}}\endset^S+\set x^{q^{2^r}+1}\endset^S=W_{2^r}$.
\end{proof}

The next question is whether or not $W_{2^s}\subseteq W_{2^r}$ when
$s\ge r$. It follows from the next result that this is never the case.

\begin{proposition}
  Let $s>r\ge 0$ be integers. Then $W_{2^r}+W_{2^s}=\ckzero{x}$.
\end{proposition}

\begin{proof}
 Let $V=W_{2^r}+W_{2^s}$. By Lemma \ref{lemma: basic for containment}
 (i), if we let $s=r+t$ with $t\ge 1$, we have
 $x^{q^{2^s}}=x^{q^{2^r2^t}}\cong x^{q^{2^r2}}=x^{q^{2^{r+1}}}\cong
 x\mod{U_{2^r}}$, and so $x^{q^{2^s}+1}\cong x^2\mod{U_{2^r}}$. But then
 $x^2\in V$. Consider first the case when $p>2$. From
 $(x+x^{q-1})^2-x^2-(x^{q-1})^2\in V$, we obtain that $2x^q\in V$ and
 since $p>2$, we obtain $x^q\in V$. On the other hand, when $p=2$, we
 observe that since $q=2^t$ for some $t\ge1$, we again obtain that
 $x^{q}\in V$. So in either case, $x^q\in V$, and thus $x^{q^{2^r}}\in
 V$. Since $x+x^{q^{2^r}}\in V$, we finally obtain $x\in V$, as
 required.
\end{proof}

Since for any $n\ge1$, $W_n$ is a proper subspace of $\ckzero{x}$, it
follows immediately that for any $r,s\ge1$ with $r\ne s$, neither of
$W_{2^r}$ and $W_{2^s}$ contains the other, and more generally, no
maximal $T$-space of $\ckzero{x}$ contains both $W_{2^r}$ and $W_{2^s}$.

From here on, $n$ shall denote a power of 2. We wish to show that for
any nonzero $f\in Y_n$, $W_n+\set f\endset^S=\ckzero{x}$.  In fact, it
suffices to consider only linear combinations of $q$-homogeneous
elements of $\ybasis{n}$; that is, it suffices to prove that if $f$ is
any nonzero $q$-homogeneous element of $Y_n$, then $W_1+\set
f\endset^S=\ckzero{x}$. 

\section{The maximality of $W_1$ in $\ckzero{x}$}

 Our objective in this section is to establish that $W_1$ is a maximal 
 $T$-space of $\ckzero{x}$. 
 
 Suppose that $X$ and $Y$ are nonempty sets with $X\subseteq Y$. We
 shall have occasion to compare the $T$-space of $\kzero{X}$
 (respectively $\ckzero{X}$) that is generated by a subset $U$ of
 $\kzero{X}$ $(\ckzero{X}$) to the $T$-space of $\kzero{Y}$
 ($\ckzero{Y}$) that is generated by the same set $U$. When necessary
 for clarity, for $U\subseteq \kzero{X}$ ($\ckzero{X}$), we shall write
 $U^S_X$, rather than $U^S$, to denote the $T$-space of $\kzero{X}$
 ($\ckzero{X}$) that is generated by $U$. Accordingly,  $U^S_Y$ would
 denote the $T$-space of $\kzero{Y}$ ($\ckzero{Y}$) generated by $U$.
 
 \begin{proposition}\label{proposition: extending t space}
  Let $X$ and $Y$ be nonempty sets with $X\subseteq Y$.
  \begin{list}{(\roman{parts})}{\usecounter{parts}} 
   \item For any $U\subseteq \kzero{X}$, $U^S_X=U^S_Y\cap \kzero{X}$.
   \item For any $U\subseteq \ckzero{X}$, $U^S_X=U^S_Y\cap \ckzero{X}$.
  \end{list}
 \end{proposition}
 
 \begin{proof}
  We shall prove the first part; the proof of the second is similar and will be omitted.
  Since every algebra endomorphism of $\kzero{X}$ extends to an algebra endomorphism of 
  $\kzero{Y}$, it follows that $U^S_X\subseteq U^S_Y$, and thus $U^S_X\subseteq U^S_Y
  \cap \kzero{X}$. It remains to prove that $U^S_Y\cap \kzero{X}\subseteq U^S_X$. 
  Let $u\in U^S_Y\cap \kzero{X}$. Then there exist $\alpha_i\in k$, $f_i\from \kzero{Y}
  \to \kzero{Y}$, $u_i\in U$, with $u=\sum \alpha_if_i(u_i)$.
  Let $g\from \kzero{Y}\to \kzero{Y}$ be the map determined by $x\mapsto
  x$ if $x\in X$, while $x\mapsto 0$ if $x\in Y-X$. As well, let
  $\iota\from \kzero{X}\to \kzero{Y}$ be  the map determined by
  $\iota(x)=x$ for each $x\in X$. Then since $u\in \kzero{X}$, we have 
  $u=g(u)=\sum \alpha_i g\comp f_i(u_i)$, and since $u_i\in U$, we have
  $u_i=\iota(u_i)$, so $u=\sum \alpha_i g\comp f_i\comp \iota(u_i)$.
  Since $g\comp f_i\comp\iota\from \kzero{X}\to\kzero{X}$, and $u_i\in U$ 
  for each $i$, it follows that $u\in U$.
 \end{proof} 

For $x\in X$, we shall make use of the homomorphism $\pi\from
\kzero{X}\to\ckzero{x}$ that is determined by sending each $z\in X$ to
$x$. For each $T$-space $V$ of $\kzero{X}$, $\pi{V}=V\cap \ckzero{x}$,
where we regard $\kzero{X}$ as a subalgebra of $\kzero{X}$ in the
natural way. This follows from the fact that $V$ is a $T$-space, and we
can consider $\pi$ as an endomorphism of $\kzero{X}$, so
$\pi(V)\subseteq V$. Thus $\pi(V)\subseteq V\cap\ckzero{x}$. For $f\in
V\cap \ckzero{x}$, $\pi(f)=f$ and so $f\in \pi(V)$, which proves that
$V\cap \ckzero{x}\subseteq \pi(V)$.

\begin{lemma}\label{lemma: move to kxy}
 Let $X$ be any set of size at least two, and let $x\in X$. For any $U\subseteq\ckzero{x}$ and $f\in \ckzero{x}$, 
 $f\in U^S$ if and only if $f\in U^{S_X}+T^{(2)}$, where $T^{(2)}$ is the commutator $T$-ideal of $\kzero{X}$ 
 (so generated by $[x,y]$ for any $y\in X$).
\end{lemma}

\begin{proof}
 Since $T^{(2)}\subseteq \ker(\pi)$, we have $\pi (U^{S_X}+T^{(2)})=\pi(U^{S_X})=U^{S_X}\cap \ckzero{x}$, and
 by Proposition \ref{proposition: extending t space}, $U^{S_X}\cap \ckzero{x}=U^S$. For $f\in \ckzero{x}$, we have
 $\pi(f)=f$, so $f\in U^{S_X}+T^{(2)}$ implies that $f=\pi(f)\in U^S$, while the converse follows from the fact
 that $U^S\subseteq U^{S_X}\subseteq U^{S_X}+T^{(2)}$.
\end{proof} 

\begin{corollary}\label{corollary: commutative in kxy}
 Let $X$ be any set of size at least two, and let $x\in X$. For any $U\subseteq\ckzero{x}$ and $f\in \ckzero{x}$, 
 $f\in U^S$ if and only if $f\in U^{S_X}$ in $\ckzero{X}$. 
\end{corollary} 
 
The following result will be very important in our work.

\begin{proposition}[\cite{fine}, Theorem 1]\label{proposition: fines result}
 Let $p$ be a prime, and let 
 \begin{alignat*}{2}
  M&=M_0+M_1p+M_2p^2+\cdots+M_tp^t &&\qquad(0\le M_r<p),\\
  N&=N_0+N_1p+N_2p^2+\cdots+N_tp^t &&\qquad(0\le N_r<p).
 \end{alignat*}
 Then
 $$
 \choice M,N \cong \choice M_0,{N_0}\choice M_1,{N_1}\choice M_2,{N_2}\cdots \choice M_t,{N_t}\mod{p}.
 $$
\end{proposition} 

We state an immediate consequence of Proposition \ref{proposition: fines result} which will be of great value 
in what follows. Recall that $k$ is a finite field of order $q$ and characteristic $p$, so
$q$ is a $p$-power.

\begin{corollary}\label{corollary: fines application}
 For any integers $t,r,i,j$ with $0\le t,r,i,j<q$,
 $$
  \choice tq+r,{jq+i}\cong \choice t,j\choice r,i\mod{p}.
 $$
\end{corollary}
  
\begin{corollary}\label{corollary: special fines result}
 Let $j$ and $t$ be integers with $0\le j\le t$. Then the following hold:
 
  \begin{list}{(\roman{parts})}{\usecounter{parts}} 
  \item If $1\le r\le q-1$ and $t\le r/2$, then modulo $p$,
  $$\choice r+t(q-1),{1+j(q-1)} \cong \begin{cases} 0 & j>1\\
  t   & j=1\\
  r-t & j=0\end{cases}
  $$
  \item If $1\le r\le q-1$ and $r+1\le t<(q+r+1)/2$, then modulo $p$,
  $$choice r+t(q-1),{1+j(q-1)} \cong \begin{cases}  \choice t-1,{j-1} \choice q+r-t,{q+1-j}& j>1\\
  t-1&j=1\\
  r-t&j=0\end{cases}
  $$
  In particular, if $1< j < t-(r-1)$, then $\choice r+t(q-1),{1+j(q-1)} \cong 0\mod{p}$.
  \item If $2\le r\le q-1$ and $t<r/2$, then modulo $p$,
  $$\choice r+t(q-1),{r-1+j(q-1)} \cong \begin{cases} 0 & j< t-1\\
  t   & j=t-1\\
  r-t & j=t\end{cases}
  $$
  \item If $2\le r\le q-1$ and $r+1\le t<(q+r+1)/2$, then modulo $p$,
  $$\choice r+t(q-1),{r-1+j(q-1)} \cong \begin{cases}
  \choice t-1,{j}\choice q+r-t,{r-1-j}& j\le r-1<t-1\\
  0 & r-1<j< t-1\\
  t-1 & j=t-1\\
  r-t & j=t\end{cases}
  $$
  \end{list}
\end{corollary}

\begin{proof}
 For the first part, we observe that $r+t(q-1)=tq+(r-t)$ with $0\le t,r-t<q$, and
 $1+j(q-1)=(j-1)q+(q+1-j)=jq+1-j$. If $j>1$, then $0\le j-1,q+1-j<q$, while if
 $j=0,1$, then $0\le j,1-j<q$. By Corollary \ref{corollary: fines application}, in the first case
 we have $\choice r+t(q-1),{1+j(q-1)} \cong \choice t,{j-1} \choice r-t,{q+1-j} \mod{p}$,
 while in the second case, we have $\choice r+t(q-1),{1+j(q-1)} \cong 
 \choice t,{j} \choice r-t,{1-j} \mod{p}$. Note that $q+1-j>r-t$ if and only if
 $q+1+t-j>r$, which holds since $r\le q-1<q+1$. Thus when $j>1$, $\choice r-t,{q+1-j}\cong 0 \mod{p}$,
 and so $\choice r+t(q-1),{1+j(q-1)} \cong 0\mod{p}$.
 
 For the second part, we observe that since $r+1\le t< (q+r+1)/2$, $0\le t-1<q+r-t\le q-1$, and so
 $r+t(q-1)=(t-1)q+(q+r-t)$ with $0\le t-1,q+r-t<q$. As well, $1+j(q-1)=(j-1)q+(q+1-j)=jq+1-j$,
 so if $j>1$, then $0\le j-1,q+1-j<q$, while if $j=0,1$, we have $0\le j,1-j<q$. In the first
 case, we obtain $\choice r+t(q-1),{1+j(q-1)} \cong \choice t-1,{j-1} \choice q+r-t,{q+1-j} \mod{p}$,
 while in the second case, we have $\choice r+t(q-1),{1+j(q-1)} \cong 
 \choice t-1,{j} \choice q+r-t,{1-j} \mod{p}$. Note that if $1<j<t-(r-1)$, then $q+1-j>q+r-t$ and
 so $\choice q+r-t,{q+1-j}\cong 0 \mod{p}$, which establishes that
 $\choice r+t(q-1),{1+j(q-1)} \cong 0\mod{p}$ when $1<j<t-(r-1)$.
 
 For (iii), we have $r+t(q-1)=tq+r-t$ with $0\le t,r-t<q$. As well, for $j\le t$, we have $r-1+j(q-1)=
 jq + r-1-j$ with $0\le j,r-(j+1)<q$ since $t<r/2$ and so $j+1\le t+1<r/2+1\le r$.
 By Corollary \ref{corollary: fines application}, $\choice r+t(q-1),{r-1+j(q-1)} \cong \choice t,{j} \choice r-t,{r-1-j}\mod{p}$.
 Since $j\le t$, we have $r-j\ge r-t$. If $j<t-1$, then $r-j-1>r-t$ and so $\choice r-t,{r-1-j}=0$. If $j=t-1$,
 then $\choice t,{j} \choice r-t,{r-1-j}=t$, and if $j=t$, then $\choice t,{j} \choice r-t,{r-1-j}=r-t$.
 
 Finally, for (iv), we have $r+t(q-1)=(t-1)q+q+r-t$ with $0\le
 t-1,q+r-t<q$. For $j\le t$, we have $r-1+j(q-1) = jq+r-1-j$ with $0\le
 j,r-1-j<q$ if $j+1\le r$, while if $j+1>r$, then we have $r-1+j(q-1) =
 (j-1)q+q+r-1-j$ with $0\le j-1,q+r-1-j<q$. Consider first the situation
 when $j+1>r$. In this case, by Corollary \ref{corollary: fines application}, 
 we have $\choice r+t(q-1),{r-1+j(q-1)}\cong 
 \choice t-1,{j-1} \choice q+r-t,{q+r-1-j}\mod{p}$. If $j<t-1$, then 
 $q+r-1-j> q+r-t$ and so $\choice q+r-t,{q+r-1-j}=0$. If $j=t-1$, then
 $\choice t-1,{j-1} \choice q+r-t,{q+r-1-j}\cong t-1\mod{p}$, while
 if $j=t$, then $\choice t-1,{j-1} \choice q+r-t,{q+r-1-j}\cong
 r-t\mod{p}$. Now suppose that $j+1\le r$. Note that $r<t$, so this
 implies that $j<t-1$. Thus $j=t-1$ or $t$ is not possible in this case.
 By Corollary \ref{corollary: fines application}, we have
 $\choice r+t(q-1),{r-1+j(q-1)}\cong \choice t-1,{j} \choice q+r-t,{r-1-j}\mod{p}$.
\end{proof}

\begin{definition}\label{definition: q homogeneous elements}
 For any $r$ with $1\le r\le q-1$, let $\class{r}$ denote the $q$-homogeneity class of $x^r$.
 Then for each $r$ with $1\le r\le q-1$, let $\ybasis{1}^{[r]}=\ybasis{1}\cap \class{r}$,
 so $\ybasis{1}^{[r]}=\set x^{jq+i}\mid q>i>j\ge 0,\ jq+i\cong r\mod{q-1}\endset$. 
\end{definition}

We shall use induction on $r$ to prove that for any $1\le r\le q-1$, and any nonzero 
$f\in Y_1\cap\class{r}$, $W_1+\set f\endset^S=\ckzero{x}$.

An element of $Y_1\cap \class{r}$ has the form 
$$
  \sum_{\substack{q>i>j\ge 0\\ j+qi\cong r\mod{q-1}}} \alpha_{i,j}x^{jq+i},
$$
where for each $i$ and $j$, $\alpha_{i,j}\in k$. Note that since
$j<i<q$, the maximum value for $jq+i$ is $(q-2)q+(q-1)= q(q-1)-1$, while
the minimum value is 1. Furthermore, $jq+i\cong r\mod{q-1}$ if and only
if $jq+i=r+t(q-1)$ for some integer $t$. Since $1\le jq+i\le q(q-1)-1$,
we would have $1\le 1+t(q-1)\le q(q-1)-1$, so $0\le t(q-1)\le q(q-1)-2$,
and thus $0\le t\le q-\frac{2}{q-1}$; that is, $0\le t\le q-1$. Of the
values of $t$ with $0\le t\le q-1$, we wish to determine those that
cause $x^{r+t(q-1)}$ to be an element of $\ybasis{1}^{[r]}$. Observe
that $r+t(q-1)= tq+r-t = (t-1)q+q+r-t$. If $r\ge t$, then
$r+t(q-1)=jq+i$ for $j=t$ and $i=r-t$, and we would then require
$q>r-t>t\ge 0$. Now, $1\le r\le q-1$, $t\ge 0$ ensure that $r-t<q$, but
$r-t>t$ if and only if $r>2t$, or $t<r/2$. Thus, of the integers $t$
with $0\le t\le r$, $x^{r+t(q-1)}\in \ybasis{1}^{[r]}$ if and only if
$0\le t<r/2$. Consider now the integers $t$ for which $r<t\le q-1$. Then
$r+t(q-1)=(t-1)q+q+r-t$ with $t-1>0$ and $q+1-(q-1)\le q+r-t< q+t-t=q$,
so $r+t(q-1)=jq+i$ for $j=t-1$ and $i=q+r-t$, and we have verified that
$q>i$, and $j\ge 0$ (in fact, $j\ge r\ge 1$). We also must have $i>j$,
and this holds if and only if $q+r-t>t-1$; that is, if and only if
$q+r+1>2t$, or $t<(q+r+1)/2$.

We have therefore established the following result.

\begin{lemma}\label{lemma: b sets}
 For each $r$ with $1\le r\le q-1$, 
 $$
  \ybasis{1}^{[r]}=\set x^{r+t(q-1)}\mid 0\le t<r/2\text{ or
   }r+1\le t<(q+r+1)/2\endset.
 $$ 
\end{lemma}  

The base case $r=1$ of our inductive argument is the content of the next lemma.

\begin{lemma}\label{lemma: base case}
 For nonzero $f\in \class{1}\cap Y_1$, $W_1+\set f\endset^S=\ckzero{x}$.
\end{lemma}

\begin{proof}
 By Lemma \ref{lemma: b sets}, $\ybasis{1}^{[1]}=\set x^{1+t(q-1)}\mid
 t=0\text{ or }2\le t<(q+2)/2\endset$. Let $f$ be a nonzero element of 
 $\class{1}\cap Y_1$. Then there exist $\alpha_t\in k$ such that 
 $$
  f=\alpha_0x+\sum_{2\le t< (q+2)/2}\alpha_tx^{1+t(q-1)}.
 $$
 By Corollary \ref{corollary: commutative in kxy}, it suffices to prove
 that $x$ belongs to the $T$-space of $\ckzero{x,y}$ that is generated
 by $\set x+x^q,x^{q+1},f\endset$; that is, that $x$ belongs to the
 $T$-space of $\ckzero{x,y}$ that is generated by $\w1,{x,y}$ and $f$.
 For convenience, let us use $W_1$ to refer to either $\w1,{x,y}$ or to
 $\w1,x$, and let $W_1+\set f\endset^S$ denote both  the $T$-space of
 $\ckzero{x,y}$ and the $T$-space of $\ckzero{x}$ that is generated by
 $\set x+x^q,x^{q+1},f\endset$. In $\ckzero{x,y}$, let $g$ denote the
 $q$-homogeneous component of $xy^{q-1}$ in $f(x+y)-f(x)\in W_1+\set
 f\endset^S$, so $g\in W_1+\set f\endset^S$ as well. We have
 \begin{align*} 
  g&=\alpha_0x+\mkern-15mu\sum_{2\le t< (q+2)/2}\mkern-20mu\alpha_t(\sum_{j=0}^t \choice 1+t(q-1),{1+j(q-1)} x^{1+j(q-1)}y^{(t-j)(q-1)}\\
    &\hskip 2in -\bigl(\alpha_0x+\mkern-15mu\sum_{2\le t< (q+2)/2}\mkern-20mu\alpha_tx^{1+t(q-1)})\bigr)\\
   &=\sum_{2\le t< (q+2)/2}\mkern-15mu\alpha_t(\sum_{j=0}^{t-1} \choice 1+t(q-1),{1+j(q-1)} x^{1+j(q-1)}y^{(t-j)(q-1)}).
 \end{align*}
 Now apply Corollary \ref{corollary: special fines result} (ii) with $r=1$ 
 (note that in particular, this gives $\choice 1+t(q-1),{1+j(q-1)}\cong 0\mod{p}$ 
 when $1<j<t$) to obtain that
 \begin{align*}
  g &=\sum_{2\le t< (q+2)/2}\alpha_t( (1-t) xy^{t(q-1)} + (t-1) x^qy^{(t-1)(q-1)})\\
 &=(xy^{q-1}-x^q)\mkern-20mu\sum_{2\le t< (q+2)/2}\alpha_t( (1-t) y^{(t-1)(q-1)}.
 \end{align*}
 Let $u=\sum_{2\le t\le (q+1)/2}\alpha_t (1-t) y^{(t-1)(q-1)}$, so
 $(xy^{q-1}-x^q)u=g\in W_1+\set f\endset^S$. Now, $x^qu\cong
 -xu^q\mod{W_1}$, so $g\cong x(y^{q-1}u+u^q)\mod{W_1}$ and thus
 $x(y^{q-1}u+u^q)\in W_1+\set f\endset^S$. Apply the endomorphism of
 $\ckzero{x,y}$ that is determined by sending $x$ to $y^{q^2-1}$ while
 fixing $y$ to obtain that  $y^{q^2-1}(y^{q-1}u+u^q)\in W_1+\set
 f\endset^S$. By Lemma \ref{lemma: almost like 1},
 $y^{q^2-1}(y^{q-1}u+u^q)\cong y^{q-1}u+u^q\mod{U_1}$, and so
 $y^{q-1}u+u^q\in W_1+\set f\endset^S$. Thus the $T$-ideal $\set
 y^{q-1}u+u^q\endset^T$ is contained in  $W_1+\set f\endset^S$. Let
 $U=U_1+\set y^{q-1}u+u^q\endset^T$. Then $U_1\subseteq U\subseteq
 W_1+\set f\endset^S$. We claim that if $y^{q-1}u+u^q\in W_1$, then
 $W_1+\set f\endset=\ckzero{x,y}$. For suppose that $y^{q-1}u+u^q\in
 W_1$. Since $u^q\cong -u\mod{W_1}$, we would then have $y^{q-1}u-u\in
 W_1$; that is, $\sum_{2\le t< (q+2)/2}\alpha_t((1-t) y^{(t)(q-1)}+
 (1-t) y^{(t-1)(q-1)})\in W_1$. Set
 $m=\left\lceil\frac{q+2}{2}\right\rceil-1$, and observe that
 $t<(q+2)/2$ means $t\le m$. We would then have
 $$
  -\alpha_2y^{q-1} +\alpha_m(1-m)y^{m(q-1)} 
  +\sum_{2\le t< m} ((1-t)\alpha_t-t\alpha_{t+1})y^{t(q-1)}\in W_1.
 $$ 
 Apply the endomorphism of $\ckzero{x,y}$ that is determined by sending 
 $y$ to $x$ while fixing $x$ to obtain that
 $$
 -\alpha_2x^{q-1}+ \alpha_m(1-m)x^{m(q-1)} +\sum_{2\le t< m} ((1-t)\alpha_t-t\alpha_{t+1})x^{t(q-1)}\in W_1.
 $$
 Note that $t(q-1)=(t-1)q+q-t$ and $q-t>t-1$ if and only if $q+1>2t$;
 that is, if and only if $t<m$. Since this condition holds in the above
 summation, $x^{t(q-1)}\in \ybasis{1}$ for every $t$ with $2\le t<m$, as
 is $x^{q-1}$. If $p$ is odd, then $m=(q+1)/2$, and then
 $m(q-1)/2=(q+1)(q-1)/2$. As $(q-1)/2$ would then be a positive integer,
 we would have $y^{m(q-1)}\in W_1$. If $p$ is even, then $m=q/2$ and
 $y^{m(q-1)}\in\ybasis{1}$. Thus $\alpha_2=0$, and for each $t$ with
 $2\le t\le m-1$, $(1-t)\alpha_t-t\alpha_{t+1}=0$, so $\alpha_t=0$ for
 each $t$ with $2\le t\le m$. But then $f=\alpha_0x$, and since $f\ne0$,
 we obtain $x\in W_1+\set f\endset^S$, as claimed. 

 It remains to consider the case when $y^{q-1}u+u^q\notin W_1$. In this
 case, $U_1\subsetneq U\subseteq W_1 +\set f\endset^S$. Suppose that
 $x\notin W_1+\set f\endset^S$. Then $U\ne\ckzero{x}$. Since $U_1$ is 
 precomplete, $U$ is also precomplete, and so by Theorem 8 of \cite{sund},
 $U=\set x-x^{q^d}\endset^T$ for some  positive integer $d$. But
 $U_1=\set x-x^{q^2}\endset^T\subsetneq U$,  so we must have $d=1$. But
 then $U=\set x-x^q\endset^{T}$, the unique maximal (actually, maximum)
 $T$-space of $\ckzero{x,y}$, which was shown in \cite{rankin} also to be a
 maximal $T$-space. Thus either $W_1+\set f\endset^S=\ckzero{x,y}$, or
 else $W_1+\set f\endset^S$ is a $T$-ideal. Suppose that $W_1+\set
 f\endset^S\ne \ckzero{x,y}$, so that $W_1+\set f\endset^S$ is a
 $T$-ideal of $\ckzero{x,y}$. Then $x^{q+1}x^i\in W_1+\set f\endset^S$
 for every positive integer $i$. In particular, $x^{q^2}\in W_1+\set
 f\endset^S$. But $x\cong x^{q^2}\mod{U_1}$ and thus modulo $W_1$, which
 means that $x\in W_1+\set f\endset^S$. As this contradicts our
 assumption that $W_1+\set f\endset^S\ne \ckzero{x,y}$, it follows that
 $W_1+\set f\endset^S=\ckzero{x,y}$, as required.
\end{proof}

\begin{proposition}\label{proposition: main inductive step}
 Let $2\le r\le q-1$. Then for any nonzero $f\in \class{r}\cap Y_1$, $W_1+\set f\endset^S=\ckzero{x}$.
\end{proposition}

\begin{proof}
 We shall prove this by induction on $r$, with the base case provided by
 Lemma \ref{lemma: base case}. Suppose that $r\ge2$, and that the result
 holds for all smaller integers. Let $f\in \class{r}\cap Y_1$ with $f\ne
 0$. By Lemma \ref{lemma: b sets}, we may assume that 
 $$
   f=\sum_{\substack{0\le t<r/2\text{ or}\\r+1\le t<(q+r+1)/2}} \alpha_t x^{r+t(q-1)}.
 $$
 Note that if $r=q-1$, then there are no indices $t$ for which $r+1\le t< (q+r+1)/2$. 

 By Corollary \ref{corollary: commutative in kxy}, it suffices to prove that $W_1+\set f\endset^S=\ckzero{x,y}$.
 In $\ckzero{x,y}$, let $g$ denote the $q$-homogeneous component of $x^{r-1}y$ in $f(x+y)$. Then
 $$
 g=\sum_{\substack{0\le t<r/2\text{ or}\\r+1\le t<(q+r+1)/2}} \alpha_t \sum_{0\le j\le t}\choice r+t(q-1),{r-1+j(q-1)} x^{r-1+j(q-1)}y^{1+(t-j)(q-1)}.
 $$
 For convenience, let $l=\left\lceil\frac{r}{2}\right\rceil-1$, so that $t<r/2$ if and only if $t\le l$. Then
 by Corollary \ref{corollary: special fines result}, (iii) for $t<r/2$ and (iv) for $r+1\le t<(q+r+1)/2$, we find that
 \begin{align*}
   g&=r\alpha_0x^{r-1}+\sum_{1\le t\le l}\alpha_t \biggl(  t x^{r-1+(t-1)(q-1)}y^{q}+ (r-t) x^{r-1+t(q-1)}y\biggr)\\ 
 &\hskip20pt+\mkern-10mu\sum_{r+1\le t<(q+r+1)/2}\mkern-10mu \alpha_t \sum_{0\le j\le r-1}\choice t-1,j\choice q+r-t,{r-1-j} x^{r-1+j(q-1)}y^{1+(t-j)(q-1)}\\
 &\hskip20pt+\mkern-10mu\sum_{r+1\le t<(q+r+1)/2} \alpha_t \biggl( (t-1) x^{r-1+(t-1)(q-1)}y^{q}
  +(r-t) x^{r-1+t(q-1)}y\biggr).
 \end{align*}
 We now apply to $g$ the endomorphism of $\ckzero{x,y}$ that is determined by sending $y$ to $x^{q^2-1}$ while fixing $x$.
 By Lemma \ref{lemma: almost like 1}, the result is congruent modulo $U_1$ to the element that is obtained by deleting $x^{q^2-1}$,
 which we shall denote by $h$. Thus, after regrouping the terms in the first summation, we find that
 \begin{align*}
   h&=(r-l)\alpha_lx^{r-1+l(q-1)}+\sum_{0\le t\le l-1}  \biggl( (r- t)\alpha_t+ (t+1)\alpha_{t+1}\biggr) x^{r-1+t(q-1)}\\
    &\hskip30pt+\sum_{r+1\le t<(q+r+1)/2} \alpha_t \sum_{0\le j\le r-1}\choice t-1,j\choice q+r-t,{r-1-j} x^{r-1+j(q-1)}\\
 &\hskip30pt+\sum_{r+1\le t<(q+r+1)/2} \alpha_t \biggl( (t-1) x^{r-1+(t-1)(q-1)}
  +(r-t) x^{r-1+t(q-1)}\biggr).
 \end{align*}
 Furthermore, since $g\in W_1+\set f\endset^S$, it follows that $h\in W_1+\set f\endset^S$ as well.
 In the first summation above, we note that $t\le l$ if and only if
 $t\le (r-1)/2$. If $l=(r-1)/2$ (possible of course only if $r$ is odd),
 then $x^{r-1+l(q-1)}=(x^{q+1})^{(r-1)/2}\in W_1$, and otherwise, $t\le
 l<(r-1)/2$ has $r-1+t(q-1)=tq+r-1-t$ with $0\le t<r-1-t<q$, so
 $x^{r-1+t(q-1)}\in \ybasis{1}$. A related observation can be made for 
 the second summation displayed above. For $0\le j\le r-1$, we find that
 $r-1+j(q-1)=jq+r-1-j$ with $0\le j, r-1-j<q$, so $x^{r-1+j(q-1)}\in
 \ybasis{1}$ if and only if $r-1-j>j$; that is, if and only if
 $j<(r-1)/2$. Observe that if $j=(r-1)/2$ (possible only when $r$ is odd
 of course), then $r-1+j(q-1)=(q+1)(r-1)/2$ and so in this case,
 $x^{r-1+j(q-1)}\in W_1$. Thus in the second summation above, we may
 exclude the value $j=(r-1)/2$. When $(r-1)/2<j\le r-1$, then
 $x^{r-1+j(q-1)}=x^{jq+r-1-j}\cong -x^{(r-1-j)q+j}
 =-x^{r-1+(r-j-1)(q-1)}\mod{W_1}$, and $(r-1)/2=r-1-(r-1)/2>r-1-j\ge 0$.
 Thus, modulo $W_1$,
 \begin{align*} 
   h&\cong (r-l)\alpha_lx^{r-1+l(q-1)}+\sum_{0\le t\le l-1}  \biggl( (r- t)\alpha_t+ (t+1)\alpha_{t+1}\biggr) x^{r-1+t(q-1)}\\
 &\hskip15pt+\mkern-20mu\sum_{r+1\le t<(q+r+1)/2} \alpha_t \sum_{0\le j\le l-1}\biggl({\textstyle\choice t-1,j\choice q+r-t,{r-1-j}} 
  -{\textstyle\choice t-1,{r-1-j}\choice q+r-t,{j}}\biggr) x^{r-1+j(q-1)}\\
 &\hskip15pt+\mkern-20mu\sum_{r+1\le t<(q+r+1)/2} \alpha_t \biggl( (t-1) x^{r-1+(t-1)(q-1)}
  +(r-t) x^{r-1+t(q-1)}\biggr).
 \end{align*}    
 For each $t$ with $r+1\le t<(q+r+1)/2$, and each $j$ with $0\le j\le l-1$, let $\beta_{t,j}=\choice t-1,j\choice q+r-t,{r-1-j} -\choice t-1,{r-1-j}\choice q+r-t,{j}$, and set
 \begin{align*} 
   h_1&=(r-l)\alpha_lx^{r-1+l(q-1)}+\sum_{0\le t\le l-1}  \biggl( (r- t)\alpha_t+ (t+1)\alpha_{t+1}\biggr) x^{r-1+t(q-1)}\\
 &\hskip40pt+\mkern-20mu\sum_{r+1\le t<(q+r+1)/2} \alpha_t\mkern-10mu \sum_{0\le j<(r-1)/2}\beta_{t,j} x^{r-1+j(q-1)}
 \end{align*}
 and 
 $$
 h_2=\sum_{r+1\le t<(q+r+1)/2} \alpha_t \biggl( (t-1) x^{r-1+(t-1)(q-1)}
  +(r-t) x^{r-1+t(q-1)}\biggr),
 $$
 so $h_1+h_2\cong h\mod{W_1}$ and thus $h_1+h_2\in W_1+\set f\endset^S$. Furthermore, we have established that if $l<(r-1)/2$,
 then $h_1$ is in the linear span of $\set u\in \ybasis{1}\mid u=x^{r-1+t(q-1)},\ 0\le t<(r-1)/2\endset$, while if $l=(r-1)/2$,
 then $(r-l)\alpha_lx^{r-1+l(q-1)}\in W_1$ and $h_1-(r-l)\alpha_lx^{r-1+l(q-1)}$ is in the linear span of
 $\set u\in \ybasis{1}\mid u=x^{r-1+t(q-1)},\ 0\le t<(r-1)/2\endset$. As for $h_2$, note that
 for $r+1\le t<(q+r+1)/2$, we have $q+r>2t$ and so $q+r-t-1>t-1$. Thus $r-1+t(q-1)=(t-1)q+(q+r-1-t)$, with $0<r\le t-1<q+r-t-1
 =q-1-(t-r)<q-1$ and so $x^{r-1+t(q-1)}\in \ybasis{1}$ for each $t$ with $r+1\le t<(q+r+1)/2$. Thus
 $h_2$ is in the linear span of
 $\set u\in \ybasis{1}\mid u=x^{r-1+t(q-1)},\ r\le t<(q+r+1)/2\endset$. Since these two subsets of $\ybasis{1}$ are disjoint, it follows that
 if either $h_1\notin W_1$ or $h_2\ne0$, then $h\cong h_1+h_2\mod{W_1}$ means that $h\ne 0$ and so $h$ is a nonzero element of 
 $\class{r-1}\cap V_1$. But then by the inductive hypothesis, $W_1+\set h\endset^S=\ckzero{x}$, and since $h\in W_1+\set f\endset^S$,
 it follows that $W_1+\set f\endset^S=\ckzero{x}$. 
 
 It remains to consider the situation when $h_2=0$ and $h_1\in W_1$; 
 that is, either $r$ is even, so $l<(r-1)/2$ and $h_1=0$, or $r$ is 
 odd, so $l=(r-1)/2$ and $h_1-(r-l)\alpha_lx^{r-1+l(q-1)}=0$. For this 
 discussion, let $m=\left\lceil\frac{q+r+1}{2}\right\rceil -1$, so 
 $r+1\le t\le m$. Then we have
 \begin{align*}
  0&=h_2=\mkern-15mu\sum_{r+1\le t\le m} \alpha_t \biggl( (t-1) x^{r-1+(t-1)(q-1)}+(r-t) x^{r-1+t(q-1)}\biggr)\\
  &=r\alpha_{r+1}x^{r-1+r(q-1)}+ (r-m)\alpha_m  x^{r-1+m(q-1)}\\
  &\hskip1in+\sum_{r+1\le t\le m-1}  \biggl( \alpha_t(r-t) +\alpha_{t+1} t \biggr)x^{r-1+t(q-1)}
 \end{align*}
 As $x^{r-1+t(q-1)}\in \ybasis{1}$ for each $t$ with $r+1\le t<(q+r+1)/2$, 
 it follows that $r\alpha_{r+1}=0$, $(r-m)\alpha_m=0$, and for each $t$ with 
 $r+1\le t\le m-1$, we have $\alpha_t(r-t) +\alpha_{t+1} t=0$. Since neither 
 $t\cong 0\mod{p}$ nor $r-t\cong 0\mod{p}$ for these values of $t$, it follows 
 that $\alpha_t=0$ for each $t$ with $r+1\le t<(q+r+1)/2$. We shall take 
 advantage of this information to dramatically simplify the presentation of
 $h_1$.  Let $l=\left\lceil\frac{r-1 }{2}\right\rceil -1$. Then
 $$
  h_1= (r-l)\alpha_lx^{r-1+l(q-1)}+\sum_{0\le t\le l-1}  \biggl( (r- t)\alpha_t+ (t+1)\alpha_{t+1}\biggr) x^{r-1+t(q-1)}
 $$
 and so $(r- t)\alpha_t+ (t+1)\alpha_{t+1}=0$ for $0\le t\le l$. Since 
 neither $r-t\cong 0\mod{p}$ nor $t+1\cong 0\mod{p}$ for any $t$ under 
 consideration, we have $\alpha_{t+1}=-\frac{(r-t)}{(t+1)}\alpha_t$ for 
 each $t$ with $0\le t\le l-1$. If $r$ is even, then $l<(r-1)/2$ and then 
 we also have $\alpha_l=0$, which then implies that $\alpha_t=0$ for every 
 $t$. Since this would imply that $f=0$, we may conclude that $r$ is odd, 
 and $l=(r-1)/2$. From the fact that $\alpha_{t+1}=-\frac{(r-t)}{(t+1)}\alpha_t$ 
 for each $t$ with $0\le t\le l-1$, we find that $\alpha_t=(-1)^t\choice r,t \alpha_0$ 
 for each $t$ with $0\le t\le (r-1)/2$, and so without loss of generality, we 
 may assume that 
 $$
  f=\sum_{0\le t\le (r-1)/2} (-1)^t\choice r,t x^{r+t(q-1)}.
 $$ 
 We shall not make use of the fact, but it may intrigue the reader to note that for $(r+1)/2\le t\le r$, we have
 $r+t(q-1)=(t-1)q+q+r-t$ and $0\le t-1,r-t$, so $x^{r+t(q-1)}\cong -x^{(q+r-t)q+t-1}\mod{W_1}$, and
 $-x^{(q+r-t)q+t-1} =-x^{q^2+(r-t)q+t-1}\cong -x^{1+(r-t)q+t-1}=-x^{(r-t)q+t}\mod{U_1}$. As $-x^{(r-t)q+t}=
 -x^{(r-t)q+t-r+r}=-x^{r+(r-t)(q-1)}$, and so 
 $$
  (-1)^t\choice r,t x^{r+t(q-1)}\cong (-1)^{t+1}\choice r,{r-t} x^{r+(r-t)(q-1)}\mod{W_1}.
 $$ 
 Since $(-1)^{t+1}=(-1)^{r-t}$, it follows that 
 $$
 \sum_{0\le t\le r} (-1)^t\choice r,t x^{r+t(q-1)}\cong
 2\sum_{0\le t\le (r-1)/2} (-1)^t\choice r,t x^{r+t(q-1)}=2f\mod{W_1}.
 $$
 Thus $2f\cong (x-x^q)^r\mod{W_1}$, and so if $p>2$, $f\cong\frac12(x-x^q)^r\mod{W_1}$.
 
 We now return to our study of $W_1+\set f\endset^S$. 
 Our work above, when specialized to the current $f$, shows that $g$, the $q$-homogeneous component of $x^{r-1}y$ in $f(x+y)$, is given by 
 $$
 g=rx^{r-1}y\mkern15mu+\sum_{1\le t\le (r-1)/2} (-1)^t \choice r,t \biggl(  t x^{r-1+(t-1)(q-1)}y^{q}+ (r-t) x^{r-1+t(q-1)}y\biggr),
 $$
 and we know that $g\in W_1+\set f\endset^S$.
 Recall that for any $u,v\in \ckzero{x,y}$, $uv^q\cong -u^qv\mod{W_1}$, so $g$ is congruent modulo $W_1$ to
 $$
  rx^{r-1}y\mkern15mu+\sum_{1\le t\le (r-1)/2} (-1)^t \choice r,t \biggl( (-t) x^{q(r-1+(t-1)(q-1))}y+ (r-t) x^{r-1+t(q-1)}y\biggr),
 $$
 so this element belongs to $W_1+\set f\endset^S$. Recall also that $U_1=\set x-x^{q^2}\endset^T\subseteq W_1$,
 so we obtain that, modulo $U_1$,
 \begin{align*}
  rx^{r-1}y\mkern5mu &+\mkern-10mu\sum_{1\le t\le (r-1)/2} (-1)^t \choice r,t \biggl( (-t) x^{q(r-t)+(t-1)q^2}y+ (r-t) x^{r-1+t(q-1)}y\biggr)\\
 &\mkern-40mu \cong y\biggl(rx^{r-1}\mkern5mu+\mkern-10mu\sum_{1\le t\le (r-1)/2} \mkern-5mu(-1)^t \choice r,t \bigl( (-t) x^{q(r-t)+(t-1)}+ (r-t) x^{r-1+t(q-1)}\bigr)\biggr),
 \end{align*}
 and thus  
 $$
 y\biggl(rx^{r-1}\mkern15mu+\sum_{1\le t\le (r-1)/2} (-1)^t \choice r,t \bigl( (-t) x^{q(r-t)+(t-1)}+ (r-t) x^{r-1+t(q-1)}\bigr)\biggr)
 $$
 belongs to $W_1+\set f\endset^S$.
 Apply the endomorphism of $\ckzero{x,y}$ that is determined by sending $y$ to $x^{q^2-1}$ while fixing $x$ to this element,
 and then apply Lemma \ref{lemma: almost like 1} to obtain that
 $$
  h=rx^{r-1}\mkern15mu+\sum_{1\le t\le (r-1)/2} (-1)^t \choice r,t \bigl( (-t) x^{q(r-t)+(t-1)}+ (r-t) x^{r-1+t(q-1)}\bigr)\biggr)
 $$
 belongs to $W_1+\set f\endset^S$.    
 Observe that
 \begin{align*}
   h&= rx^{r-1}\mkern15mu+\sum_{1\le t\le (r-1)/2} (-1)^t \choice r,t \bigl( (-t) x^{q(r-t)+(t-1)}+ (r-t) x^{r-1+t(q-1)}\bigr)\\
   &= rx^{r-1} + rx^{q(r-1)} \mkern15mu+\sum_{2\le t\le (r-1)/2} (-1)^t \choice r,t (-t) x^{q(r-t)+(t-1)}\\
   &\hskip50pt+\mkern-20mu \sum_{1\le t\le (r-3)/2}\mkern-15mu (-1)^t{\choice r,t} (r-t) x^{r-1+t(q-1)}\\
   &\hskip70pt+ (-1)^{(r-1)/2}{ \choice r,{(r-1)/2}\frac{r-1}{2}} (x^{q+1})^{(r-1)/2}.
 \end{align*}
 Now, 
 $$ \sum_{2\le t\le (r-1)/2}\mkern-20mu (-1)^t{\textstyle \choice r,t} (-t) x^{q(r-t)+(t-1)}
  = \mkern-20mu\sum_{1\le t\le (r-3)/2}\mkern-20mu (-1)^{t+1}{\textstyle \choice r,{t+1}} (-1)(t+1) x^{q(r-t-1)+t}.
 $$ 
 As well, 
  $$
  (-1)^{t+1} {\textstyle\choice r,{t+1}} (-1)(t+1)=(-1)^{t}r!/(t!(r-t-1)!)=(-1)^t{\textstyle\choice r,t} (r-t).
  $$
  Thus with
  $$
   \beta=(-1)^{(r-1)/2}{\textstyle\choice r,{(r-1)/2}(r-1)/2}\not\cong 0\mod{p},
  $$
  we have
 \begin{align*}
  h&=rx^{r-1} + rx^{q(r-1)} +\sum_{1\le t\le (r-3)/2} (-1)^t \choice r,t (r-t)( x^{q(r-t-1)+t}\\
  &\hskip2.5in + x^{tq+r-1-t}) + \beta (x^{q+1})^{(r-1)/2}\\
   &=\sum_{0\le t\le (r-3)/2} (-1)^t \choice r,t (r-t)( x^{q(r-t-1)+t}
   + x^{q(t)+(r-t-1)}) + \beta (x^{q+1})^{(r-1)/2}.
 \end{align*}
   Now, $0\le t\le (r-1)/2\le (q-3)/2$ means that $q>q-3\ge r-t-1\ge (r-1)/2\ge t\ge 0$ and $r-t-1+t=r-1>0$, so
   $h$ is in the linear span of $\set x^{qi+j}+x^{i+qj}\mid q>i\ge j\ge 0,\ i+j>0\endset$, and since $h\ne 0$,
   it follows that $h\in W_1-U_1$. As well, we have $yh\in W_1+\set f\endset^S$, so $\set h\endset^T\subseteq W_1+\set f\endset^S$, and $h\notin U_1$
   means that $U_1\subsetneq U=U_1+\set h\endset^T\subseteq W_1+\set f\endset^S$. As in the proof of
   Lemma \ref{lemma: base case}, this implies that $U=\set x-x^q\endset^T$ and $W_1+\set f\endset^S=\ckzero{x,y}$, as required. This completes the proof of the inductive
   step, and so the result follows.
\end{proof}

\begin{theorem}
 $W_1$ is a maximal $T$-space of $\ckzero{x}$.
\end{theorem}

\begin{proof}
 We must prove that for any $f\in \ckzero{x}-W_1$, $W_1+\set f\endset^S=\ckzero{x}$. As observed in the discussion
 following Definition \ref{definition: yn def}, it suffices to prove this for each $q$-homogeneous $f\in Y_1$, the 
 linear span of $\ybasis{1}$. But each $q$-homogeneous element of $Y_1$ is in the class of $\class{r}$ for some $r$ with
 $1\le r\le q-1$. The result follows now from Lemma \ref{lemma: base case} for $r=1$, and from Proposition 
 \ref{proposition: main inductive step} for $2\le r\le q-1$.
\end{proof}

\section{Summary}
 We have shown that for any prime $p$, and any finite field $k$ of characteristic $p$ and order $q$, the $T$-spaces 
 $W_{2^n}=\set x+x^{q^{2^n}},x^{q^{2^n}+1}\endset^S$, $n\ge 0$ are proper, and for any $0\le m<n$, $W_{2^m}+W_{2^n}=\ckzero{x}$.
 We have also proven that $W_1$ is maximal. In \cite{rankin}, for $p>2$, we had proven that the $T$-spaces $\set x+x^{q^{2^n}}\endset^S$,
 $n\ge0$, were proper and had the property that for any $0\le m<n$, $\set x+x^{q^{2^m}}\endset^S
 +\set x+x^{q^{2^n}}\endset^S=\ckzero{x}$, and so were able to conclude that $\ckzero{x}$ had infinitely many
 maximal $T$-spaces. From our knowledge of the $k$-linear basis for $W_{2^n}$ that we have obtained in this paper,
 it follows that $x^{q^{2^n}+1}\notin \set x+x^{q^{2^n}}\endset^S$, so none of the $T$-spaces $\set x+x^{q^{2^n}}\endset^S$
 are maximal in $\ckzero{x}$. For $p=2$, the situation is somewhat different. Also in \cite{rankin}, we had proven that for $p=2$,
 the family of $T$-spaces $\set x+x^{q},x^{q^{2^n}+1}\endset^S$, $n\ge0$, were proper and had the property that the sum
 of any two is $\ckzero{x}$. But for $p=2$, we have $W_{2^n}=\set x+x^{q^{2^n}},x^{q^{2^n}+1}\endset^S\subseteq \set x+x^{q},x^{q^{2^n}+1}\endset^S$,
 and also from our knowledge of a basis for $W_{2^n}$, we may observe that $x+x^q\notin W_{2^n}$ for $n>0$. For
 $n=0$, the two $T$-spaces coincide, and we have proven that $W_{2^0}$ is a maximal $T$-space of $\ckzero{x}$. 
 It seems possible that for $p>2$, $W_{2^n}$ is a maximal $T$-space of $\ckzero{x}$ for every $n\ge0$, and for $p=2$,  
 $\set x+x^{q},x^{q^{2^n}+1}\endset^S$ is a maximal $T$-space of $\ckzero{x}$ for each $n\ge0$.


\begin{thebibliography}{0}

\bibitem{rankin} C. Bekh-Ochir and S. A. Rankin, S. A., {\em Maximal $T$-spaces of a free associative algebra}, J. Algebra, 332 (2011), 442--456.

\bibitem{fine} N. J. Fine, {\em Binomial coefficients modulo a prime}, Amer. Math. Monthly 54 (1947), 589--592.

\bibitem{grI} A. V. Grishin, {\em On the finite-basis property of systems of generalized polynomials}, Izv. Math. USSR, {\bf 37}, no. 2, 1991, 243--272.

\bibitem{grII} A. V. Grishin, {\em On the finite-basis property of abstract $T$-spaces}, Fund. Prikl. Mat., {\bf 1}, 1995, 669--700 (Russian).

\bibitem{sund} T. R. Sundararaman, {\em Precomplete varieties of $R$-algebras}, Algebra Universalis 3 (1975), 397--405.
\end{thebibliography}
\end{document}